\documentclass{amsart}

\usepackage{mathtools, amssymb, amsthm} 
\usepackage{bm}  
\usepackage{kotex}  

\usepackage{tikz}
\usetikzlibrary{angles, arrows.meta, calc, positioning, quotes}
\usepackage{hyperref}
\usepackage{float}  

\usepackage{xcolor}

\newcommand{\vcz}{\bm{\mathrm{0}}}
\newcommand{\vc}[1]{\bm{\mathrm{#1}}}
\newcommand{\dd}{\mathrm{d}}

\newtheorem*{lemma}{Lemma}

\tikzset{
  vechead/.style={-{Stealth[scale=0.8]}}
}

\begin{document}

\title[A Simple Geometric Proof of Kepler’s First Law]{A simple Geometric Proof of Kepler’s First Law via Velocity Projection Ratio}
\author{Hyounggyu Choi}
\address{SNU College, Seoul National University, Seoul, Republic of Korea}
\email{hgchoi66@snu.ac.kr}
\date{\today}

\begin{abstract}
We present a simple geometric proof of Kepler’s first law. The core of the argument rests on a kinematic invariant that we call the velocity projection ratio—the ratio between the radial projection of the velocity and its projection onto a fixed direction parallel to the periapsis line. We show that the constancy of this ratio is an immediate consequence of Newton’s laws, and that it leads directly to the focus–directrix property of the orbit. By establishing a direct link between planetary dynamics and the classical geometric definition of conics, our proof reveals the conic nature of the orbit as an inevitable consequence of the inverse-square law.
\end{abstract}
\maketitle

\section{Introduction}

The fact that a particle moving under an inverse-square central force traces a conic section has been proved in several classical ways since Isaac Newton. 
Classical derivations of Kepler’s First Law largely fall into several well-established traditions.

First is Newton’s original geometric argument in the \emph{Principia}, where the inverse-square force 
is combined with polygonal approximations of orbital motion to show that the orbit must be a conic. 
Although historically profound, this method is technically demanding and relies on a number of intricate geometric lemmas \cite{NewtonPrincipia}.

A second family of proofs uses the modern analytic approach, which involves solving a second-order differential equation for the orbit. 
This approach appears in standard mechanics texts such as Goldstein, Poole, and Safko, and also in Landau–Lifshitz, yielding the familiar conic expression 
$\displaystyle r(\theta) = \frac{p}{ 1 + \epsilon \cos\theta}$. 
While systematic, it requires solving differential equations and relating several orbital parameters \cite{GoldsteinCM, LandauLifshitz, Whittaker}.

A third classical route uses the hodograph method introduced by Hamilton. 
The hodograph—the curve traced by the tip of the velocity vector in velocity space—is a circle for Keplerian motion. 
This method provides deep geometric insight into the velocity structure but requires an additional step to transform the velocity-space circle back into the position-space conic \cite{HamiltonHodograph}.

A fourth approach utilizes the Laplace–Runge–Lenz (LRL) vector. 
The conservation of this vector implies that the motion lies on a conic, with the eccentricity encoded in the vector's magnitude. 
This method is elegant but relies on specific vector identities and the unique algebraic symmetry of the inverse-square law \cite{Runge1919, Lenz1924, GoldsteinLRL}.

A fifth, more recent tradition seeks to bridge the gap between Newton’s geometric intuition and modern vector algebra without solving differential equations. 
Among the various efforts to derive Kepler's laws without differential equations, the approach by Provost and Bracco \cite{Provost2009} stands out for its elegant geometric insight. By demonstrating that the velocity projection onto the transverse direction $\hat{\theta}$ is proportional to the inverse of the radial distance, they reconstructed the orbit as a conic. 
This approach is highly instructive and provides a brilliant pedagogical bridge between Newtonian dynamics and orbital geometry. Their proof is revisited in the appendix of this paper.

While we acknowledge the elegance of the aforementioned traditions, our approach offers a more direct and self-contained path. 
The key idea originates from classical pursuit problems, such as the \emph{$n$-bug problem}, where each bug continuously chases the next. In such problems, the rate at which the distance between the chaser and the target closes is determined solely by the projections of the chaser's velocity and target's velocity onto the line connecting them. Inspired by this, we consider two dynamically meaningful projections of the planet’s velocity: one onto the radial direction and one onto a fixed direction parallel to the periapsis line. Remarkably, the ratio of these two projections—the \textit{velocity projection ratio}—remains constant throughout the motion, satisfying the focus-directrix property of a conic and directly revealing the orbit as a conic.
The focus-directrix property means that, for every point on the conic, the distance $r$ to its focus is equal to $\epsilon$ times its distance $d$ to the directrix:
\[
r = \epsilon d.
\]
Here, $\epsilon$ is the eccentricity of the conic.
This approach not only yields a concise proof but also exposes a direct structural connection between Newtonian dynamics and the classical geometry of conics.

\section{Proof of Kepler's First Law}

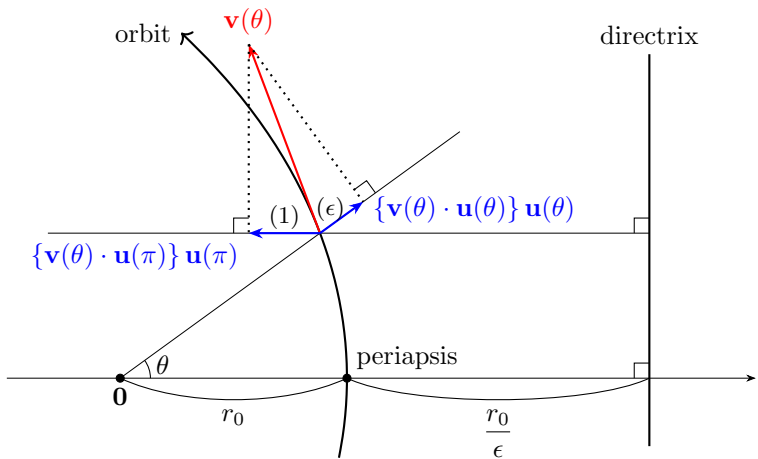
\begin{figure}[h!]
\begin{center}
\begin{tikzpicture}[scale=3]
  \def \rightanglesymsize {0.1} 

  \def \gm {0.5}  
  \def \e {0.75}  
  \def \rz {1}   
  
  \def \vz {sqrt((1 + \e) * \gm / \rz)}  
  \def \l {(1 + \e) * \rz}  

  \coordinate (O) at (0, 0);
  \filldraw[black] (O) circle (0.5pt) node[anchor=north] {$\vcz$};

  \draw[very thin,vechead] (-0.5, 0) -- ({1.2 * (1 + 1/\e) * \rz}, 0) coordinate (G);

  \filldraw[black] (\rz, 0) circle (0.5pt) node[anchor=south west] {periapsis};
  \draw[thin] (0, 0) to [out=335,in=205,distance=0.3cm] 
    node[anchor=midway,below] {$r_0$} (\rz, 0);

  \draw[thick,->] plot [variable=\t,domain=-20:80,smooth,samples=20] 
    ({\l / (1 + \e * cos(\t)) * cos(\t)}, {\l / (1 + \e * cos(\t)) * sin(\t)}) node[anchor=east] {orbit};

  \foreach \phi in {pi/5} {
      \coordinate (p) at ({\l / (1 + \e * cos(\phi r)) * cos(\phi r)}, {\l / (1 + \e * cos(\phi r)) * sin(\phi r)});

      \draw[thick] ({(1 + 1/\e) * \rz}, -0.3) -- ({(1 + 1/\e) * \rz}, 0) coordinate (I) -- ({(1 + 1/\e) * \rz}, 1.43) coordinate (H) node[anchor=south] {directrix};
      \draw[thin] (\rz, 0) .. controls +(335:0.3) and +(205:0.3) ..
        node[anchor=midway,below] {$\displaystyle\frac{r_0}{\epsilon}$} ({(1 + 1 / \e) * \rz}, 0);
    
      \draw[very thin] (0, 0) -- ($1.7*(p)$) coordinate (C);

      \draw[very thin] ($(p)+(-1.2,0)$) coordinate (F) -- ({(1 + 1/\e) * \rz}, {\l / (1 + \e * cos(\phi r)) * sin(\phi r)}) coordinate (J);
      
      \coordinate (v) at ({(\gm / (\rz * \vz) * (-sin(\phi r))) }, {(\gm / (\rz * \vz) * (\e + cos(\phi r))) });
      \draw[red,thick,vechead] (p) -- ++(v) node[anchor=south,xshift=0,yshift=0] {$\vc{v}(\theta)$};
    
      \def \radialcomponent {((\gm / (\rz * \vz) * (-sin(\phi r))) * cos(\phi r) + (\gm / (\rz * \vz) * (\e + cos(\phi r))) * sin(\phi r))}
      \draw[blue,thick,vechead] (p) -- node[black,anchor=midway,left,xshift=5,yshift=3] {$\text{\small ($\epsilon$)}$} ++({\radialcomponent * cos(\phi r)}, {\radialcomponent * sin(\phi r)}) 
        node[anchor=west,yshift=-3] {$\{\vc{v}(\theta) \cdot \vc{u}(\theta)\}\,\vc{u}(\theta)$};
    
      \def \horizontalcomponent {((\gm / (\rz * \vz) * (sin(\phi r)))}
      \draw[blue,thick,vechead] (p) -- node[black,anchor=midway,above,yshift=-2] {$\text{\small ($1$)}$} ++({-\horizontalcomponent}, 0) node[anchor=north east] {$\{\vc{v}(\theta) \cdot \vc{u}(\pi)\}\,\vc{u}(\pi)$};
    
      \draw[thick,dotted] ($(p)+(v)$) coordinate (A) -- ($(p)+({\radialcomponent * cos(\phi r)}, {\radialcomponent * sin(\phi r) })$) coordinate (B);
      \draw[thick,dotted] ($(p)+(v)$) coordinate (D) -- ($(p)+({-\horizontalcomponent}, 0)$) coordinate (E);
      
      \draw pic [draw,angle radius=2mm] {right angle = A--B--C};
      \draw pic [draw,angle radius=2mm] {right angle = D--E--F};
      \draw pic [draw, angle radius=2mm] {right angle = H--I--O};
      \draw pic [draw, angle radius=2mm] {right angle = H--J--p};

      \draw pic [draw,angle radius=4mm] {angle = G--O--C};
      \draw (O) node[anchor=south west,xshift=10,yshift=-2] {$\theta$};
  }
\end{tikzpicture}

\caption{Kepler's first law: The proof is completely encoded in the fact that the ratio
$\displaystyle\frac{\vc{v}(\theta)\cdot\vc{u}(\theta)}{\vc{v}(\theta)\cdot\vc{u}(\pi)}=\epsilon$
is independent of $\theta$.}
\end{center}
\end{figure}

We place the mass $M$ at the origin and denote the Newtonian gravitational constant by $G$.
The planetary trajectory is written in polar form as
\[
\vc{r}(\theta) = r(\theta)(\cos\theta,\sin\theta),
\]
or simply
\[
\vc{r}= r(\cos\theta,\sin\theta),
\]
where the polar angle $\theta$ is a function of time $t$.
Let $\vc{v}=\vc{v}(\theta)$ denote the velocity at the point $\vc{r}(\theta)$.

From the \emph{conservation of angular momentum}, the specific angular momentum
\[
h = |\vc r \times \vc v|
\]
is constant. We do not consider the case $h=0$.
Hence,
\[
\frac12 r^{2}\,\dd\theta
   = \frac12 h \,\dd t ,
\]
which yields
\begin{equation}\label{eq:areal_temp}
\frac{\dd\theta}{\dd t} = \frac{h}{r^2},~
\text{ or }~~\frac{\dd t}{\dd\theta} = \frac{r^2}{h}.
\end{equation}
Since $\dd\theta/\dd t$ never vanishes, the polar angle $\theta$ is a strictly monotone function of time.
Without loss of generality, we have assumed that $\theta$ is strictly increasing.

\emph{Newton's law of gravitation} gives
\[
\frac{\dd \vc{v}}{\dd t} = -\frac{GM}{r^2}\vc{u}(\theta),
\quad \text{where} \quad
\vc{u}(\theta) = (\cos\theta,\sin\theta).
\]
Using~\eqref{eq:areal_temp}, we obtain
\[
\frac{\dd\vc{v}}{\dd\theta}
= \frac{\dd\vc{v}}{\dd t}\frac{\dd t}{\dd\theta}
= -\frac{GM}{r^2}\vc{u}(\theta)\frac{r^2}{h}
= -\frac{GM}{h}\vc{u}(\theta).
\]
Integrating with respect to $\theta$, we find
\[
\begin{aligned}
\vc{v}(\theta)
&= \vc{v}(0) + \int_0^\theta - \frac{GM}{h}\vc{u}(u)\,\dd u \\
&= \vc{v}(0) - \frac{GM}{h}\int_0^\theta (\cos u,\sin u)\,\dd u \\
&= \vc{v}(0) - \frac{GM}{h}(\sin\theta,1-\cos\theta) \\
&= \vc{v}(0) + \frac{GM}{h}\bigl\{(0,-1)+\vc{u}(\theta+\pi/2)\bigr\}.
\end{aligned}
\]
Thus we obtain an \emph{explicit expression for the velocity vector} as a function of the polar angle:
\begin{equation}\label{eq:v-theta_temp}
\vc{v}(\theta)
= \vc{v}(0) + \frac{GM}{h} \bigl\{(0,-1)+\vc{u}(\theta+\pi/2)\bigr\}.
\end{equation}
We note that the above expression is $2\pi$-periodic in $\theta$.

As a preliminary step, we establish the existence of a periapsis without assuming any \emph{a priori} geometric form of the orbit. At the periapsis, the position and velocity vectors are inherently orthogonal; securing its existence is therefore a critical prerequisite that justifies our subsequent kinematic formulation.

\begin{lemma}
The orbit admits a periapsis.

\end{lemma}

\begin{proof}
From~\eqref{eq:v-theta_temp}, it is easy to see that the speed $|\vc{v}(\theta)|$ is bounded above.
By the conservation of mechanical energy,
\[
\frac12|\vc v(\theta)|^2 - \frac{GM}{r(\theta)} = E,
\]
it follows that the radial distance $r(\theta)$ is bounded below by a positive number.

Since the polar angle $\theta$ is a monotone function of time $t$, its range
\[
I_\theta := \{\theta(t)\mid -\infty<t<\infty\}
\]
is an open connected interval.
Consequently, $I_\theta$ must be of one of the following forms:
\[
I_\theta = (-\infty,\infty),\quad (a,\infty),\quad (-\infty,b),\quad \text{or}\quad (a,b).
\]

Suppose first that $I_\theta$ is unbounded on at least one side.
Then $I_\theta$ contains a compact subinterval of length $2\pi$.
By~\eqref{eq:v-theta_temp}, the velocity $\vc v(\theta)$ is a $2\pi$-periodic
vector-valued function of~$\theta$, and hence the speed $|\vc v(\theta)|$
is also $2\pi$-periodic.
Therefore, $|\vc v(\theta)|$ attains a global maximum on any compact subinterval
of length~$2\pi$. By the conservation of mechanical energy, the radial distance $r(\theta)$ is minimized at those points where $|\vc v(\theta)|$ is maximized.
Thus, in this case the orbit admits a periapsis.

Now assume that $I_\theta=(a,b)$. Then $\displaystyle\lim_{t\to\infty}\theta(t)=b$ and, by the continuity of $\vc{v}(\theta)$, we see $\displaystyle\lim_{t\to\infty}\vc{v}(\theta)=\vc{v}(b)$.\\
If $\vc{v}(b)\neq\vc{0}$, the velocity approaches a nonzero constant vector, so the motion becomes asymptotically uniform, and hence
\[
\lim_{t\to\infty}r=\infty.
\]
If $\vc{v}(b)=\vc{0}$, then, since $\displaystyle h=|\vc{r}\times\vc{v}|\le r|\vc{v}|$, we have $\displaystyle r\ge\frac{h}{|\vc{v}|}\longrightarrow\infty$.

Thus, in either case,
\[
\lim_{t\to\infty}r=\infty.
\]
An analogous argument shows that $\displaystyle\lim_{t\to -\infty} r = \infty$. It follows that the radial distance $r$ diverges at both temporal ends and
therefore attains a global minimum at some finite time.
Hence, the orbit admits a periapsis.

Specifically, the cases $I_\theta=(a,\infty)$ and $I_\theta=(-\infty,b)$ cannot occur.
\end{proof}

After a suitable time shift and a suitable choice of coordinate axes, we may assume that the periapsis is located at
\[
\vc{r}(0) = (r_0,0).
\]
At the periapsis, the velocity $\vc{v}(0)$ is orthogonal to the position vector.
Thus we may write
\[
\vc{v}(0) = (0,v_0),
\]
and consequently
\[
h = |\vc{r}(0) \times \vc{v}(0)| = r_0 v_0.
\]
Accordingly, since $(0,1)=\vc{u}(\pi/2)$, equation~\eqref{eq:v-theta_temp} reduces to
\begin{equation}\label{eq:v-theta}
\vc{v}(\theta)
= \frac{r_0 v_0^{2}-GM}{r_0 v_0}\,\vc{u}(\pi/2)
+ \frac{GM}{r_0 v_0}\,\vc{u}(\theta+\pi/2).
\end{equation}

We now compare two projections of the velocity vector:
the \emph{radial projection} $\{\vc{v}(\theta)\cdot \vc{u}(\theta)\}\,\vc{u}(\theta)$, and 
the \emph{horizontal projection} $\{\vc{v}(\theta)\cdot \vc{u}(\pi)\}\,\vc{u}(\pi)$.\\
Since $\vc{u}(\pi/2)\cdot\vc{u}(\theta)=\sin\theta$ and $\vc{u}(\theta+\pi/2)\cdot\vc{u}(\theta)=0$,
\[
\vc{v}(\theta)\cdot\vc{u}(\theta)
= \frac{r_0 v_0^{2}-GM}{r_0 v_0}\sin\theta.
\]
Since $\vc{u}(\pi/2)\cdot \vc{u}(\pi)=0$ and $\vc{u}(\theta+\pi/2)\cdot \vc{u}(\pi)=\sin\theta$,
\[
\vc{v}(\theta)\cdot \vc{u}(\pi) 
=  \frac{GM}{r_0 v_0}\sin\theta.
\]

The \emph{ratio of signed projections} is therefore
\begin{equation}\label{eq:def-eccentricity}
\frac{\vc{v}(\theta)\cdot\vc{u}(\theta)}
     {\vc{v}(\theta)\cdot \vc{u}(\pi)}
 = \frac{r_0 v_0^{2} - GM}{GM}
 =: \epsilon.
\end{equation}
As is evident from the expression, this ratio $\epsilon$ is independent of~$\theta$.

Since $\displaystyle \epsilon = \frac{r_0{v_0}^2 - GM}{GM}$ and $h=r_0v_0$, equation \eqref{eq:v-theta} can be written as
\begin{equation}\label{eq:v-theta_epsilon}
\vc{v}(\theta)
= \frac{GM}{h} \bigl\{ \epsilon \,\vc{u}(\pi/2) + \vc{u}(\theta+\pi/2)\bigr\}.
\end{equation}
It follows that
\[
|\vc{v}(\theta)|^2
= \left(\frac{GM}{h}\right)^2
\left(\epsilon^2+1+2\epsilon\cos\theta\right).
\]
Since $\theta=0$ is the periapsis, $|\vc{v}(\theta)|^2$ must attain its maximum at $\theta=0$. This is possible only if 
\[
\epsilon\geq0.
\]

Note that the case $\epsilon=0$ corresponds to a constant speed. Together with conservation of angular momentum, this implies that the radial distance remains constant, $r(\theta)=r_0$, and hence the orbit is circular.

For the case $\epsilon > 0$, we now introduce a \emph{directrix} to bridge our kinematic result with classical geometry. Let the directrix be the line perpendicular to the periapsis line, located on the side opposite the central mass at a distance $\displaystyle r_0/\epsilon$ from the periapsis. By this construction, the ratio of the distances from the periapsis to the central mass ($r_0$) and to the directrix ($r_0/\epsilon$) is exactly $\epsilon:1$.

This setup allows us to recover the \emph{focus--directrix property of the orbit} directly from the velocity projections. Crucially, the radial projection of the velocity, $\vc{v}(\theta) \cdot \vc{u}(\theta)$, represents the rate of change of the distance $r$ to the central mass, $\dot{r}=\dd r/\dd t$. Simultaneously, the projection onto the fixed direction, $\vc{v}(\theta) \cdot \vc{u}(\pi)$, represents the rate of change of the distance $d$ to the chosen directrix, $\dot{d}=\dd d/\dd t$. 

Our analysis established that the ratio of these two projections is the constant $\epsilon$ throughout the entire motion:
\begin{equation*}
\dot{r} = \epsilon\dot{d}.
\end{equation*}
By construction of the directrix, the distance relation $r = \epsilon d$ is satisfied at the periapsis. Hence, the identity
\begin{equation*}
r = \epsilon d
\end{equation*}
holds throughout the motion. This is precisely the focus-directrix property of a conic section with focus at the origin and eccentricity $\epsilon$. Thus, the orbit is a conic section, completing the proof of Kepler’s first law.

\section{Appendix}
\subsection{Classical Structures Hidden in the Velocity Representation}
\newcommand{\Hodograph}[2]{
  \begin{tikzpicture}[scale=2.8]
    \def \gm {0.5} 
    \def \e {#1}  
    \def \rz {1}   
    
    \def \vz {sqrt((1 + \e) * \gm / \rz)} 
    
    \def \tz {#2 * 0.2}
    
    \draw[vechead] (-0.7, 0) -- (0.7, 0);
    \draw[vechead] (0, -0.2) node[anchor=north] {$(\epsilon=#1)$} -- (0, 1.2) coordinate (A);
    
    \draw[blue,very thick] plot [variable=\t,domain=-#2:#2,smooth,samples=20] 
    ({(\gm / (\rz * \vz) * (-sin(\t)))}, {(\gm / (\rz * \vz) * (\e + cos(\t)))});
    
    \foreach \i [evaluate=\i as \phi using #2*\i] in {-1,-0.8,...,1} {
        \draw[thin,vechead] (0, 0) -- ({(\gm / (\rz * \vz) * (-sin(\phi)))}, {(\gm / (\rz * \vz) * (\e + cos(\phi)))});
    }
    
    \filldraw[black] (0, {\gm / (\rz * \vz) * \e}) coordinate (C) circle (0.3pt) node[anchor=west] {$\frac{GM}{r_0v_0}(0,\epsilon)$};
    
    \draw[blue] (0, {\gm / (\rz * \vz) * \e}) -- ++($1.2 * \gm / (\rz * \vz)*({-sin(\tz)}, {cos(\tz)})$) coordinate (B);
    \draw[red,thick,vechead] (0, 0) -- ({(\gm / (\rz * \vz) * (-sin(\tz)))}, {(\gm / (\rz * \vz) * (\e + cos(\tz)))}) node[anchor=east] {$\vc{v}(\theta)$};
    
    \draw pic [red,draw,angle radius=3mm] {angle = A--C--B};
    \draw (C) node[red,anchor=south east,xshift=2,yshift=6] {$\theta$};
    
    \draw[thick,blue] (0, {\gm / (\rz * \vz) * \e}) to [out=335,in=25,distance=0.5mm] node[anchor=midway,right] {$(\epsilon)$} (0, 0);
    \draw[thick,blue] (0, {\gm / (\rz * \vz) * \e}) to [out=25,in=335,distance=0.7mm] node[anchor=midway,right] {$(1)$} (0, {\gm / (\rz * \vz) * (\e + 1)});
  \end{tikzpicture}
}

In this appendix, we demonstrate that the explicit algebraic form \eqref{eq:v-theta} of the velocity vector $\mathbf{v}$ naturally encompasses the geometry of the velocity hodograph and the Laplace-Runge-Lenz vector.

\subsubsection{Circularity of the Hodograph}
The \emph{hodograph} of a motion is defined as the curve traced by the tip of the velocity vector $\vc{v}(\theta)$ in velocity space as the parameter $\theta$ varies.
Hamilton famously showed that the hodograph of Keplerian motion is a circle (or an arc of a circle). The expression~\eqref{eq:v-theta_epsilon} makes clear that the hodograph is a circle of radius 
$\displaystyle \frac{GM}{h}$, centered at 
$\displaystyle \left(0,\frac{GM\epsilon}{h}\right)$ in velocity space; 
nothing further needs to be explained.

\begin{figure}[h!]
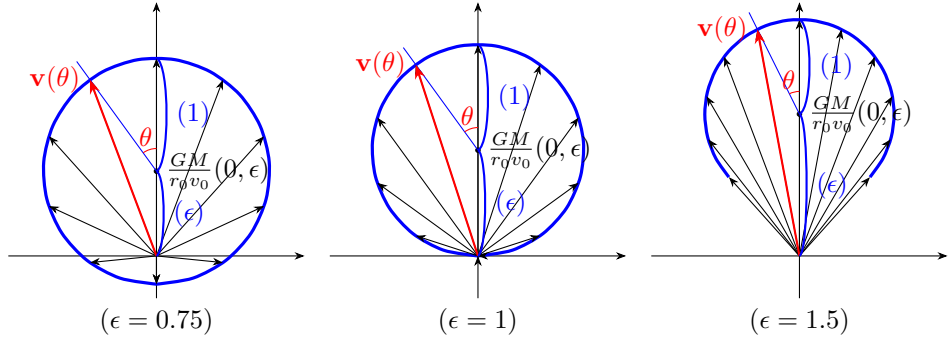

    \Hodograph{0.75}{180}
    \Hodograph{1}{180}
    \Hodograph{1.5}{acos(-1/\e)}
    \caption{Hodographs for the cases $\epsilon=0.75$, $\epsilon=1$, and $\epsilon=1.5$.}
\end{figure}

\subsubsection{Constancy of the Laplace--Runge--Lenz vector}

Let the mass of the planet be $m$. The \emph{Laplace--Runge--Lenz (LRL) vector} is defined by
\[
\vc{A} := \vc{p}\times\vc{L} - mk\hat{r},
\]
where
\[
k = GMm,\qquad \vc{p} = m\vc{v},\qquad \vc{L} = \vc{r}\times\vc{p}
\]
denote the gravitational parameter, the linear momentum, and the angular momentum, respectively. This vector is a famous constant of motion.
Since $\hat{r}=\vc{u}(\theta)$, rewriting this expression in terms of the velocity, we obtain
\[
\vc{A}
= m^2\bigl\{\vc{v}\times(\vc{r}\times\vc{v}) - GM\vc{u}\bigr\}.
\]

Here we give an alternative and simple verification of its constancy based on the explicit velocity representation~\eqref{eq:v-theta}. Since the overall factor $m^2$ is irrelevant for constancy, it suffices to show that
\[
\vc{v}\times(\vc{r}\times\vc{v}) - GM\vc{u}
\]
is independent of $\theta$.
First, by conservation of angular momentum,
\[
\vc{r}\times\vc{v} = h \,\hat{z},
\]
where $\hat{z}$ is the unit vector normal to the orbital plane.
Hence
\[
\vc{v}\times(\vc{r}\times\vc{v})
= h\, \vc{v}\times\hat{z}.
\]

Observing that, within the orbital plane, the cross product with $\hat{z}$ corresponds to a rotation by $-\pi/2$, we may write
\[
\vc{v}\times\hat{z} = R_{-\pi/2}(\vc{v}).
\]
Therefore, using the explicit expression \eqref{eq:v-theta},
\[
\begin{aligned}
h\, \vc{v}\times\hat{z}
&= h\, R_{-\pi/2}(\vc{v}) \\
&= h\, R_{-\pi/2}\!\left(\frac{GM}{h} \bigl\{ \epsilon \,\vc{u}(\pi/2) + \vc{u}(\theta+\pi/2)\bigr\}\right) \\
&= GM\bigl\{\epsilon \,\vc{u}(0) + \vc{u}(\theta)\bigr\}.
\end{aligned}
\]
Subtracting the term $GM\vc{u}$, we obtain
\[
\vc{v}\times(\vc{r}\times\vc{v}) - GM\vc{u}
= GM\epsilon \,\vc{u}(0)
= GM(\epsilon,\,0),
\]
which is manifestly constant.
This proves the constancy of the Laplace--Runge--Lenz vector.

\subsection{Provost and Bracco} 
Provost and Bracco~\cite{Provost2009} already offered an elegant perspective, revealing that the geometry of Keplerian motion is deeply encoded in velocity projections. By highlighting the geometric meaning of the transverse velocity component, their approach provides an insightful connection between angular momentum conservation and orbital geometry. To make these insights more transparent and seamlessly connect them to our present framework, let us humbly summarize their work in our own terms. 

They focus on the transverse projection of the velocity vector, 
\[ \bigl(\vc{v}\cdot\hat{\theta}\bigr)\hat{\theta}, \]
where, in our notation, $ \hat{\theta}=\vc{u}(\theta+\pi/2)$.  Its connection with the areal velocity is immediate. Indeed, conservation of angular momentum gives 
\[ 
r\bigl(\vc{v}\cdot\hat{\theta}\bigr)=h, 
\] 
and hence 
\begin{equation}\label{eq:r_l/(inner prod)}
r=\frac{h}{\vc{v}\cdot\hat{\theta}}\,. 
\end{equation}
Using equation~\eqref{eq:v-theta_epsilon}, we can evaluate $\vc{v}\cdot\hat{\theta}$ directly: 
\[ 
\vc{v}\cdot\hat{\theta} = \frac{GM}{h} \left(1+\epsilon\cos\theta\right).
\] 
Combining this with equation~\eqref{eq:r_l/(inner prod)}, we obtain 
\[ 
r = \frac{h^2/GM} {1+\epsilon\cos\theta}. 
\]
Thus, the transverse velocity projection identified by Provost and Bracco leads directly to the familiar polar equation of a conic. This completes the reconstruction of their argument in our notation.

\section*{Epilogue}
\begin{quote}
\small
\textbf{Author's Note:} \textit{This epilogue is of an interpretative and historical nature, offered solely as a personal reflection on a lifelong quest, and does not constitute part of the formal mathematical proof.}
\end{quote}

After years of relentless computation, Kepler declared his first law: that the orbit of Mars is an ellipse with the Sun at one focus. Working in the early seventeenth century, armed only with observational data and relentless numerical computation, Kepler asserted this law not as a theorem but as a hard-won conviction. Decades later, Isaac Newton, inheriting Kepler’s profound insight, succeeded in providing a mathematical proof. His argument was profoundly geometric and undeniably intricate. Subsequent generations offered alternative perspectives: Johann Bernoulli through differential equations, William Rowan Hamilton through the elegant hodograph method, and modern physics through the constancy of the Laplace--Runge--Lenz vector.

At this point, one may ask a natural historical question: who contributed most decisively to our understanding of planetary motion? Many would answer Newton. Few would dispute the grandeur of his achievement. His unification of the Earth and the heavens remains one of the greatest intellectual triumphs in human history, deserving timeless admiration.

Yet Kepler's accomplishment possesses a different kind of greatness. Newton supplied the proof; Kepler supplied the vision. Working in an era when neither modern mechanics nor the necessary mathematical language yet existed, Kepler proclaimed a law that no one could have anticipated from observation alone. Newton's mathematical success was possible because Kepler had already recognized the hidden geometric order of planetary motion.

To celebrate Newton's triumph need not diminish the originality of Kepler's insight. If Newton transformed a profound conjecture into mathematical certainty, Kepler first created that conjecture through extraordinary scientific imagination. In this sense, Kepler was not merely a precursor to Newton but the indispensable architect of the modern understanding of planetary motion.

Ultimately, Gauss could determine the orbit of Ceres because he knew Kepler's laws; Kepler determined the orbit of Mars even though he did not know Kepler's laws.

This historical reflection naturally leads to a deeper philosophical question:
Why should planetary orbits be conic sections at all?
Long, long before Kepler, Greek geometers such as Euclid and Apollonius studied curves obtained by slicing cones with planes, apparently for reasons internal to geometry with no practical motivation. That these same curves should later emerge as the laws governing the heavens is nothing short of remarkable. Is this coincidence to be attributed to divine intervention or teleological design?

In this article, we have presented a proof of Kepler’s first law that is distinguished by its brevity and geometric transparency. By working directly with intrinsic kinematic quantities, our argument makes clear that the appearance of conic sections is not the result of external design, but a geometric necessity. There is no hidden intention in the shape of an orbit. If any command was ever issued to the heavens, it was not an act of design, but a single kinematic mandate imposed by the inverse–square law:
\begin{center}
    \textit{``Maintain the projection ratio.''}
\end{center}
Following this single kinematic mandate inevitably carves a conic section out of the void. The beauty of planetary paths emerges not from a pre-ordained blueprint but as an inevitable consequence of the deep structural harmony between dynamics and geometry.

\section*{Conflict of Interest}
The author declares that there is no conflict of interest regarding the publication of this paper.


\end{document}